%% file: main.tex
\documentclass[a4paper,12pt]{amsart}
\usepackage[left=2cm,top=2.5cm,right=2cm,bottom=2.5cm]{geometry}
\usepackage{amsmath}
\usepackage{amssymb}
\usepackage{amsfonts}
\usepackage{amsthm}
\usepackage{stmaryrd}
\usepackage{url}
\usepackage{hyperref}
\usepackage{fancyhdr}
\usepackage{mathrsfs}
\usepackage{mathtools}
\usepackage[demo]{graphicx}
\usepackage{color}
\usepackage{enumerate,mdwlist}

\usepackage{amstext}
\usepackage{array}

\usepackage{tikz}
\usetikzlibrary{matrix,positioning,arrows,decorations.pathmorphing}
\usetikzlibrary{tikzmark}

\usepackage[shortlabels]{enumitem}
\setlist[enumerate]{leftmargin=25pt}
\setlist[itemize]{leftmargin=25pt}

\newtheorem{thm}{Theorem}[section]
\newtheorem{lemma}[thm]{Lemma}
\newtheorem{prop}[thm]{Proposition}

\newtheorem{conj}[thm]{Conjecture}

\newtheorem{ques}{Question}

\theoremstyle{definition}

\theoremstyle{definition}
\newtheorem{defn}[thm]{Definition}

\newcommand{\supp}{\ensuremath{\operatorname{supp}}}
\newcommand{\Tr}{\ensuremath{\operatorname{Tr}}}

\newcommand{\bra}[1]{{\left({#1}\right)}}

\newcommand{\scal}[1]{{\left\langle{#1}\right\rangle}}
\newcommand{\set}[1]{{\left\{{#1}\right\}}}

\DeclarePairedDelimiter\abs{\lvert}{\rvert}

\newcommand{\de}{\ensuremath{\delta}}

\newcommand{\ze}{\ensuremath{\zeta}}

\newcommand{\La}{\ensuremath{\Lambda}}
\newcommand{\la}{\ensuremath{\lambda}}
\newcommand{\sig}{\ensuremath{\sigma}}

\newcommand{\om}{\ensuremath{\omega}}
\newcommand{\Om}{\ensuremath{\Omega}}

\newcommand{\ZZ}{\ensuremath{\mathbb{Z}}}

\newcommand{\FF}{\ensuremath{\mathbb{F}}}
\newcommand{\QQ}{\ensuremath{\mathbb{Q}}}
\newcommand{\RR}{\ensuremath{\mathbb{R}}}

\newcommand{\wt}[1]{\ensuremath{\widetilde{#1}}}
\newcommand{\wh}[1]{\ensuremath{\widehat{#1}}}

\newcommand{\CC}{\ensuremath{\mathbb{C}}}

\newcommand{\Gal}{\ensuremath{\textnormal{Gal}}}

\newcommand{\sm}{\ensuremath{\setminus}}
\newcommand{\ssq}{\ensuremath{\subseteq}}

\newcommand{\vn}{\ensuremath{\varnothing}}

\newcommand{\ol}{\ensuremath{\overline}}

\newcommand{\bs}[1]{\ensuremath{\mathbf{#1}}}

\newcommand{\out}[1]{}

\definecolor{redi}{RGB}{255,38,0}
\definecolor{redii}{RGB}{200,50,30}
\definecolor{yellowi}{RGB}{255,251,0}
\definecolor{bluei}{RGB}{0,150,255}
\definecolor{blueii}{RGB}{135,247,210}
\definecolor{blueiii}{RGB}{91,205,250}
\definecolor{blueiv}{RGB}{115,244,253}
\definecolor{bluev}{RGB}{1,58,215}
\definecolor{orangei}{RGB}{240,143,50}
\definecolor{yellowii}{RGB}{222,247,100}
\definecolor{greeni}{RGB}{166,247,166}

\tikzset{ 
table/.style={
  matrix of nodes,
  row sep=-\pgflinewidth,
  column sep=-\pgflinewidth,
  nodes={rectangle,draw=black,text width=1.25ex,align=center},
  text depth=0.25ex,
  text height=1ex,
  nodes in empty cells
  },
texto/.style={font=\footnotesize\sffamily},
title/.style={font=\small\sffamily}
}

\numberwithin{equation}{section}

\title[A linear programming approach to Fuglede's conjecture in $\ZZ_p^3$]{A linear programming approach\\ to Fuglede's conjecture in $\ZZ_p^3$}



\author{Romanos Diogenes Malikiosis} 
\address{Aristotle University of Thessaloniki, Department of Mathematics, 541 24 Thessaloniki, Greece}
\email{romanos@math.auth.gr}

\begin{document}
\begin{abstract}
We present an approach to Fuglede's conjecture in $\ZZ_p^3$ using linear programming bounds, obtaining the following partial result: if $A\ssq\ZZ_p^3$ with $p^2-p\sqrt{p}+\sqrt{p}<\abs{A}<p^2$, then $A$ is not spectral.
\end{abstract}
\maketitle

\input{Sections/Introduction}

\input{Sections/SpectralTile}

\input{Sections/Delsarte}

\input{Sections/FourierProjective}

\input{Sections/Proof}

\input{Sections/Future}

 \bigskip\bigskip

\bibliographystyle{amsplain}
\bibliography{mybib}

\end{document}

%% file: Sections/Introduction.tex
\bigskip\bigskip
\section{Introduction}\label{sec:intro}
\bigskip\bigskip

A fundamental question in harmonic and functional analysis predicts that a geometric property of a measurable set (\emph{tiling}) coincides with an analytic one (\emph{spectrality}). This was posed initially by Fuglede \cite{Fuglede} in the setting of finite dimensional Euclidean spaces, following a question of Segal in the commutativity of certain partial differential operators. 
\begin{conj}\label{conj:Fuglede}
Let $\Om\ssq\RR^d$ be a bounded measurable set with $m(\Om)>0$. Then, $\Om$ tiles $\RR^d$ by translations (i.~e. almost every $x\in\RR^d$ can be written uniquely as $\om+t$, where $\om\in\Om$ and $t$ belongs to a fixed set of translations $T$) if and only if $L^2(\Om)$ accepts an orthogonal basis of complex exponential functions, $e^{2\pi i\scal{\la,x}}$, for $\la\in\La$ ($\La$ is called the spectrum of $\Om$).
\end{conj}
While Conjecture \ref{conj:Fuglede} has been largely disproved, namely for $d\geq3$ \cite{FMM06, KMhadamard, TaoFuglede}, there is a great recent interest in the setting of finite Abelian groups. Besides, even Fuglede himself hinted at different settings for Conjecture \ref{conj:Fuglede}.

Let $G$ be a finite Abelian group. We say that a subset $A\ssq G$ tiles $G$ by translations (or just that $A$ is a \emph{tile}) if there is a subset $T\ssq G$ such that each element of $G$ can be expressed uniquely as $a+t$, with $a\in A$, $t\in T$. This fact will be denoted as $A\oplus T=G$, and $T$ is called the tiling complement of $A$. In general, we will denote by $A+T$ the sumset $\set{a+t:a\in A, t\in T}$; if every element of $A+T$ has a unique representation as a sum of an element of $A$ and an element of $T$, we emphasize this by writing $A\oplus T$.

We may also define the following inner product on the space of complex functions defined on $A\ssq G$:
\[\scal{f,g}_A=\sum_{x\in A}f(x)\ol{g(x)}.\]
$\wh{G}$ denotes as usual the dual group of $G$, i.~e. the group of multiplicative characters $\chi:G\to S^1$. We call $B\ssq\wh{G}$ a \emph{spectrum} of $A$ if $\abs{A}=\abs{B}$ and $\scal{\chi,\psi}_A=0$, $\forall\chi\neq\psi\in B$; in other words, the characters of $B$, when restricted to $A$, form an orthogonal basis in $\CC^A$. In this case, $A$ is called \emph{spectral} and $(A,B)$ a spectral pair.

There is a connection between the discrete and continuous setting of Fuglede's conjecture, as was first suggested by \L{}aba \cite{Laba}, who connected this question to the tiling results by Coven \& Meyerowitz in $\ZZ$ \cite{CM}, and Tao, in his disproof of Fuglede's conjecture in $\RR^d$, $d\geq5$ \cite{TaoFuglede}. The principal argument in Tao's paper is the following:
\begin{center}
\emph{a counterexample in $G$ with $d$ generators can be lifted to a counterexample in $\RR^d$}
\end{center}
and this holds for both directions, namely the spectral to tiling and tiling to spectral directions, denoted henceforth as \textbf{S-T} and \textbf{T-S}, respectively. Tao found a spectral subset in $\ZZ_3^5$ which has 6 elements \cite{TaoFuglede}; such a subset cannot be a tile, as the cardinality of a tile divides the order of the group.

Counterexamples in $\ZZ_8^3$ \cite{KMhadamard} and $\ZZ_{24}^3$ \cite{FMM06} were given shortly thereafter, which led to counterexamples in $\RR^3$ in both directions; since \textbf{S-T} and \textbf{T-S} are hereditary properties, these counterexamples show that Fuglede's conjecture fails in $\RR^d$, $d\geq3$, in both directions.

The original conjecture is still open in $\RR$ and $\RR^2$. The counterexamples mentioned above, raised naturally the following question\footnote{This question is usually referred to as the \emph{discrete Fuglede conjecture}.}: 
\begin{ques}\label{ques:discreteFC}
For which finite Abelian groups $G$ do the properties of tiling and spectrality coincide?
\end{ques}

In the last few years, there is a wealth of results on Fuglede's conjecture in the discrete setting (there are of course the notable results \cite{Qpfuglede2, Qpfuglede, LM19} where this conjecture has been confirmed in $\QQ_p$ and for all convex domains in $\RR^d$). So far, all the results are positive in cyclic groups \cite{KMSV22, LL22a, LL22b, LL22c, M22, Zhang22} or groups with two generators \cite{FKS22, Zhang21}. A counterexample in a group of the form $\ZZ_N^2$ would also be lifted to a counterexample in $\RR^2$, whereas the confirmation of Fuglede's conjecture in every cyclic group would be an enormous step to confirming this conjecture in $\RR$ \cite{DL14} (the only missing ingredient would then be the \emph{rationality of spectrum} \cite{BoseMadan, Laba2intervals}).

For groups with at least three generators, it is natural to search for counterexamples, as we know that Fuglede's conjecture fails in $\RR^3$. Using the fundamental theorem on the factorization of finite Abelian groups, a counterexample in $\ZZ_p^d$ for all primes $p$ can be lifted to a counterexample to a finite Abelian group of $\geq d$ generators. This endeavor has so far been very successful in the \textbf{S-T} direction \cite{FS2020}:

\begin{thm}
\textbf{S-T} fails in $\ZZ_p^4$ for every odd prime $p$, and also in $\ZZ_2^{10}$.
\end{thm}
This implies that \textbf{S-T} fails in every Abelian group with $\geq10$ generators, as well as every group with $\geq4$ generators having odd order.
For the \textbf{T-S} direction there no other counterexamples known besides $\ZZ_{24}^3$ (and every group that contains a copy thereof), which led to the following conjecture \cite{Shi19a}:

\begin{conj}
\textbf{T-S} holds in every finite Abelian $p$-group.
\end{conj}
Summarizing, the state of the art regarding groups of the form $\ZZ_p^d$ is the following:
\begin{enumerate}
    \item Fuglede's conjecture is true if $d\leq2$ \cite{IMP15}.
    \item For $d\geq4$ and $p$ odd prime, \textbf{S-T} fails \cite{FS2020, Mattheus}.
    \item \textbf{S-T} fails in $\ZZ_2^{10}$ \cite{FS2020}.
    \item Fuglede's conjecture holds in $\ZZ_2^6$ \cite{FS2020}.
    \item \textbf{T-S} holds in $\ZZ_p^3$ \cite{Atenetal}.
    \item \textbf{S-T} holds in $\ZZ_p^3$, and $p\leq7$ \cite{FMV19}.
\end{enumerate}
For $p=2$, the question is still open for $\ZZ_2^d$, $d= 7, 8, 9$, and the \textbf{T-S} direction for all $d\geq7$, whereas for odd $p$ the \textbf{S-T} direction is unknown only for $d=3$ (and the \textbf{T-S} direction for all $d\geq4$). 

Here, we will approach the \textbf{S-T} direction in $\ZZ_p^3$ using linear programming bounds. This method provides estimates for the size of sets $B$ whose set of differences $B-B$ avoids a certain ``forbidden set". The method of linear programming bounds has seen widespread applications in coding theory \cite{Delsarte, KL78}, combinatorial designs \cite{dOV10}, discrete geometry, most notably in the sphere packing problem by Viazovska \cite{Packing24, V17}, who used the version by Cohn and Elkies \cite{CohnElkies}, which is also known as ``Delsarte's method" \cite{Delsarte}.

This method involves the notion of a \emph{positive definite} function, which is not new in this conjecture \cite{KMMS22, LM19}. The novelty of the present approach is the notion of a \emph{blocking set} in the finite projective plane, whose size estimates lead to the main result:

\begin{thm}\label{mainthm}
Let $A\ssq\ZZ_p^3$ be a set satisfying $p^2-p\sqrt{p}+\sqrt{p}<\abs{A}<p^2$. Then $A$ is not spectral.
\end{thm}
It has been shown \cite{FMV19} that a spectral subset $A\ssq\ZZ_p^3$ which is neither a singleton nor the entire group, must have $pk$ elements, where $1\leq k\leq p$. If $k=1$ or $p$, then $A$ is also a tile, so we may restrict our attention to the range $1<k<p$. In \cite{FMV19} it was also shown that if $k=p-2$ or $p-1$, then $A$ cannot be spectral, so Theorem \ref{mainthm} extends this result.

The paper is organized as follows: in Section \ref{SpectralTile} we summarize all the known facts about spectral subsets and tiles in $\ZZ_p^3$; we also include the \textbf{T-S} proof in $\ZZ_p^3$ to make this article as self-contained as possible. In Section \ref{Delsarte} we introduce the first main tool for the proof of Theorem \ref{mainthm}, whereas in Section \ref{FourierProjective} we introduce the second, which involves the geometric structure of the finite projective plane. Finally, in Section \ref{Proof} we are ready to prove Theorem \ref{mainthm}, and in Section \ref{Future} we discuss several possible extensions of the approach presented herein.

%% file: Sections/SpectralTile.tex
\bigskip\bigskip
\section{Spectral sets and tiles in \texorpdfstring{$\ZZ_p^3$}{}}\label{SpectralTile}
\bigskip\bigskip
All results and arguments (or generalizations thereof) in this section have already appeared in \cite{Atenetal, FMV19, IMP15}. As usual, the Fourier transform of $f:G\to\CC$ is denoted by $\hat{f}$ or $\bs Ff$ and satisfies
\[\hat{f}(\xi)=\sum_{x\in G}f(x)\xi(-x)=\scal{f,\xi},\]
for every $\xi\in\wh{G}$. With this definition, the Fourier inversion formula is
\begin{equation}\label{eq:Fourier_inversion}
f(x)=\frac{1}{\abs{G}}\sum_{\xi\in\wh{G}}\hat{f}(\xi)\xi(x),
\end{equation}
and the Plancherel-Parseval formula becomes
\[\abs{G}\sum_{x\in G}\abs{f(x)}^2=\sum_{\xi\in\wh{G}}\abs{\hat{f}(\xi)}^2.\]
The convolution $f*g$ is defined as
\[f*g(x)=\sum_{y\in G}f(y)g(x-y),\]
and the Fourier transform converts convolution to pointwise product, i.~e.
\[\wh{f*g}=\hat{f}\cdot\hat{g}.\]
The operator $\bs U=\frac{1}{\sqrt{\abs{G}}}\bs F$ is unitary, and $\bs U^2$ preserves all even functions, that is, $\hat{\hat{f}}=\abs{G}f$ for every even function $f:G\to\CC$. Also, if we denote by $f_-$ the function $f_-(x)=f(-x)$, then $\abs{\wh{f*f_-}}=\abs{\hat{f}}^2$.

It is a fundamental fact that $G$ is isomorphic to $\wh{G}$, albeit non-canonically. For $G=\ZZ_p^3$, we fix such an isomorphism $x\mapsto\xi_x$ for each $x=(x_1,x_2,x_3)\in\ZZ_p^3$, where
\[\xi_x(y)=\ze_p^{\scal{x,y}}.\]
Here, $\ze_p=e^{2\pi i/p}$ and
\[\scal{x,y}=x_1y_1+x_2y_2+x_3y_3, \;\;\; \text{ for }y=(y_1,y_2,y_3)\in\ZZ_p^3.\]
Under this identification, we write $\hat{f}(x)$ instead of $\hat{f}(\xi_x)$, and each character $\xi_x$ with $x\neq0$ is an epimorphism from $\ZZ_p^3$ to the multiplicative group of the $p$-th roots of unity; the kernel of $\xi_x$ is precisely the set of all $y$ that are ``orthogonal'' to $x$, that is, $\scal{x,y}=0$. When $\ZZ_p^3$ is viewed as a $3$-dimensional space over the finite field $\ZZ_p$, the kernel is then a plane orthogonal to $x$, which we denote by $x^\perp$. A fact that challenges our intuition from the Euclidean space $\RR^3$ is that we may have $\scal{x,x}=0$, or equivalently, we may have $V+V^\perp\neq\ZZ_p^3$ for subspaces $V\ssq\ZZ_p^3$, but this won't be an obstacle to our investigations.

We may also identify the spectrum of a spectral set $A\ssq\ZZ_p^3$ with a subset of $\ZZ_p^3$ itself. So, $B\ssq\ZZ_p^3$ is a spectrum of $A$ if $\scal{\xi_b,\xi_{b'}}_A=0$, for every pair $b, b'$ of distinct elements in $B$. Hence,
\[0=\scal{\xi_b,\xi_{b'}}_A=\sum_{a\in A}\xi_b(a)\ol{\xi_{b'}(a)}=\sum_{x\in A}\bs1_A(x)\xi_{b-b'}(x)=\wh{\bs1}_A(b'-b),\]
where the function $\bs1_A$ denotes the indicator function on the set $A$,
yielding a basic fact on spectra.
\begin{prop}\label{prop:basic}
A set of pairwise orthogonal characters, restricted on $A$ and indexed by $B\ssq\ZZ_p^3$, satisfies 
\[B-B:=\set{b-b':b,b'\in B}\ssq Z(\wh{\bs1}_A)\cup\set{0},\]
where $Z(f)=\set{x\in\ZZ_p^3:f(x)=0}$ is the zero set of the function $f$. In addition, $\abs{A}=\abs{B}$ holds, if and only if $B$ is a spectrum of $A$.
\end{prop}
Next we will study the main properties of spectral sets, especially the nontrivial ones, i.~e. those that are neither singletons nor the entire group. First of all, any nontrivial spectral set $A\ssq\ZZ_p^3$ must satisfy $Z(\wh{\bs1}_A)\neq\vn$. Therefore, the zero set $Z(\wh{\bs1}_A)$ is a union of \emph{punctured lines}, i.~e. one-dimensional subspaces of $\ZZ_p^3$ with the origin missing. Indeed, if $x\in Z(\wh{\bs1}_A)$ and $\la\in\ZZ_p^*$ is arbitrary, then
\[\wh{\bs1}_A(\la x)=\sum_{a\in A}\ze_p^\scal{\la x,-a}=\sum_{a\in A}\sig(\ze_p^\scal{x,-a})=\sig\bra{\sum_{a\in A}\ze_p^\scal{x,-a}}=\sig(\wh{\bs1}_A(x))=0,\]
where $\sig\in\Gal(\QQ(\ze_p)/\QQ)$ which satisfies $\sig(\ze_p)=\ze_p^\la$. We usually denote by $L$ and $H$ lines and planes, and by $L^*$, $H^*$ their punctured counterparts, when $L$ and $H$ contain the origin.

The structure of the vanishing sums of $p$-th roots of unity\footnote{A very special case of the main result in \cite{Redei54}.} is the next crucial tool towards a first characterization of spectral subsets of $\ZZ_p^3$.
\begin{prop}\label{prop:vanishingpsums}
If $a_0,\dotsc,a_{p-1}$ are integers such that
\begin{equation}\label{eq:vanishingpsums}
a_0+a_1\ze_p+\dotsb+a_{p-1}\ze_p^{p-1}=0,
\end{equation}
then $a_0=a_1=\dotsb=a_{p-1}$.
\end{prop}

\begin{proof}
Taking all possible Galois automorphisms on both sides of \eqref{eq:vanishingpsums}, we get
\[a_0+a_1\ze_p^\la+\dotsb+a_{p-1}\ze_p^{\la(p-1)}=0, \;\;\; \forall 1\leq\la\leq p-1,\]
or equivalently, the vector $\bs a=(a_0,\dotsc,a_{p-1})$ is orthogonal to each of the $p-1$ vectors $(1,\ze_p^\la,\dotsc,\ze_p^{\la(p-1)})$. Therefore, it must be parallel to the all-1 vector, $\bs1=(1,1,\dotsc,1)$, proving the desired fact.
\end{proof}

A simple consequence of Proposition \ref{prop:vanishingpsums} is the fact that $x\in Z(\wh{\bs1}_A)$ if and only if $A$ is equidistributed with respect to the $p$ planes parallel to $x^\perp$. Indeed, if we denote
\[H_j=\set{y\in\ZZ_p^3:\scal{x,y}=j},\]
we have
\[0=\wh{\bs1}_A(x)=\sum_{a\in A}\ze_p^\scal{x,a}=\sum_{j=0}^{p-1}\abs{A\cap H_j}\ze_p^j,\]
hence $\abs{A\cap H_0}=\abs{A\cap H_1}=\dotsb=\abs{A\cap H_{p-1}}$.
\begin{thm}\label{thm:charspec}
Let $A\ssq\ZZ_p^3$ be a nontrivial spectral set. Then $\abs{A}=pk$, with $1\leq k\leq p$. Moreover,
\begin{enumerate}
    \item if $\abs{A}=p$ or $p^2$, then $A$ is a tile.
    \item\label{empty} if $1<k<p$ (assuming that such a spectral set exists), then $Z(\wh{\bs1}_A)$ intersects every punctured plane, but contains none.
\end{enumerate}
\end{thm}

\begin{proof}
Since we have shown that $\wh{\bs1}_A$ vanishes at some $x\in\ZZ_p^3$, $A$ must be equidistributed with respect to the $p$ planes parallel to $x^\perp$. This shows that $p\mid\abs{A}$. 

If $\abs{A}>p^2$, then the same must hold for a spectrum $B$. Let $x\neq0$ be arbitrary, and consider the intersections of $B$ with the planes parallel to $x^\perp$; at least one of them must contain at least $p+1$ elements. But then, $B-B$ intersects all punctured lines in $x^\perp$. Indeed, let $b_0,\dotsc,b_p\in B$ such that $\scal{b_j,x}$ is constant for $0\leq j\leq p$. Then, the $p+1$ elements of $B-b_0$, $0, b_1-b_0,\dotsc,b_p-b_0$, are all in $x^\perp$; therefore, for every line $L\ssq x^\perp$, there is a parallel one $L'$ in $x^\perp$ that contains at least two elements of $B-b_0$. The difference of these two elements belongs to $(B-B)\cap L^*$. Since $x$ is arbitrary, and $L\ssq x^\perp$ is arbitrary as well, the above argument shows that $B-B$ intersects every punctured line, hence Proposition \ref{prop:basic} yields that $\wh{\bs1}_A$ vanishes on every nonzero element of $\ZZ_p^3$, and then \eqref{eq:Fourier_inversion} shows that $\bs1_A(x)=\frac{\abs{A}}{p^3}$ for all $x\in\ZZ_p^3$, which can only hold if $A=\ZZ_p^3$, contradicting the fact that $A$ is nontrivial. Thus, $\abs{A}=pk$, with $1\leq k\leq p$.

Suppose now that $k=1$; since $\wh{\bs1}_A$ vanishes at some $x\in\ZZ_p^3$, $A$ must be equidistributed with respect to the planes parallel to $H=x^\perp$. which shows that every such plane must contain exactly one element of $A$. It is then obvious that $A\oplus H=\ZZ_p^3$, and $A$ is a tile.

Next, suppose that $k=p$. If $A-A$ intersects every punctured line, then by Proposition \ref{prop:basic} with the roles of $A$ and $B$ reversed, $\wh{\bs1}_B$ would vanish everywhere except for the origin, which would yield again $B=\ZZ_p^3$ as above, a contradiction. So, there is a line $L$ through the origin such that $(A-A)\cap L^*=\vn$, hence $A+L=A\oplus L$. Since $\abs{A}=p^2$ and $\abs{L}=p$, we will have $A\oplus L=\ZZ_p^3$, showing that $A$ is a tile.

Finally, assume that there exists a spectral set $A$ with $1<k<p$. If there existed a plane $H$ through the origin such that $Z(\wh{\bs1}_A)\cap H^*=\vn$, then $(B-B)\cap H=\set{0}$ for a spectrum $B$ of $A$, by Proposition \ref{prop:basic}. But then no two elements of $B$ would be congruent modulo the subgroup $H$, which would imply $\abs{B}\leq\abs{\ZZ_p^3/H}=p<\abs{A}$, a contradiction. Thus, $Z(\wh{\bs1}_A)$ intersects every punctured plane.

If there existed a plane $H$ through the origin such that $H^*\ssq Z(\wh{\bs1}_A)$, the plane itself would satisfy $H-H=H\ssq Z(\wh{\bs1}_A)\cup\set{0}$, which would imply that the $\abs{H}=p^2$ characters $\xi_x$, for $x\in H$, would be pairwise orthogonal, when restricted to $A$. This is impossible, as the dimension of complex functions on $A$ has cardinality $\abs{A}<p^2$. Thus, $Z(\wh{\bs1}_A)$ contains no punctured plane, completing the proof.
\end{proof}

We end this Section with the proof that \textbf{T-S} holds in $\ZZ_p^3$.
\begin{thm}
Every tile $A\ssq\ZZ_p^3$ is spectral.
\end{thm}

\begin{proof}
 Let $T$ be a tiling complement of $A$. $A\oplus T=\ZZ_p^3$ written in functional notation is $\bs1_A*\bs1_T=\bs1_{\ZZ_p^3}$. Applying the Fourier transform on both sides we get $\wh{\bs1}_A\wh{\bs1}_T=p^3\bs1_0$, hence 
\begin{equation}\label{eq:exclusivesupport}
\supp(\wh{\bs1}_A)\cap\supp(\wh{\bs1}_T)=\set{0}.
\end{equation}
We have already seen that both supports are unions of punctured lines, along with the origin. If the support of $A$ is the entire group $\ZZ_p^3$, then $\supp(\wh{\bs1}_T)=\set{0}$, which implies that $T=\ZZ_p^3$ and $A$ a singleton, and vice versa, if $\supp(\wh{\bs1}_T)=\ZZ_p^3$, we get $A=\ZZ_p^3$. So, if we assume that $A$ is nontrivial, so that $\abs{A}=p$ or $p^2$, we get that both $\wh{\bs1}_A$ and $\wh{\bs1}_T$ must vanish somewhere. 

Suppose first that $\abs{A}=p$. Consider a punctured line $L^*\ssq Z(\wh{\bs1}_A)$. Then, the line $L$ is a spectrum of $A$ by Proposition \ref{prop:basic}, as $\abs{A}=\abs{L}$ and $L-L=L\ssq Z(\wh{\bs1}_A)\cup\set{0}$.

Finally, suppose that $\abs{A}=p^2$, hence $\abs{T}=p$. Consider a line $L$ through two points of $T$; now let a plane $H$ through the origin that is orthogonal to the direction of $L$. For any $x\in H^*$, we must have $\wh{\bs1}_T(x)\neq0$ by Proposition \ref{prop:vanishingpsums}, since $T$ is not equidistributed with respect to the planes parallel to $x^\perp$ (the one containing $L$ has at least 2 elements of $T$). Therefore, \eqref{eq:exclusivesupport} yields $H^*\ssq Z(\wh{\bs1}_A)$, and $H$ is a spectrum of $A$ by Proposition \ref{prop:basic}, since $H-H=H\ssq Z(\wh{\bs1}_A)\cup\set{0}$ and $\abs{H}=\abs{A}$.
\end{proof}

%% file: Sections/Delsarte.tex
\bigskip\bigskip
\section{Delsarte's method and balanced functions}\label{Delsarte}
\bigskip\bigskip
Regardless whether $A\ssq\ZZ_p^3$ is spectral or not, we may ask how many pairwise orthogonal characters exist, when they are restricted to $A$. Of course, there are at most $\abs{A}$ many such characters, with equality if and only if $A$ is spectral. In general, Proposition \ref{prop:basic} shows that if we have a set of characters, indexed by $B\ssq\ZZ_p^3$, such that $\scal{\xi_b,\xi_{b'}}_A=0$ for every $b,b'\in B$ with $b\neq b'$, then the difference set $B-B$ must avoid $\supp(\wh{\bs1}_A)\sm\set{0}$, which may be viewed as a ``forbidden set".

\begin{defn}
Let $E\ssq G$, such that $0\in E=-E$. We call $h:G\to \RR$ a \emph{witness function} with respect to $E$, if it satisfies the following properties:
\begin{enumerate}
\item $h$ is an even function, such that $h(x)\leq0$, $\forall x\in G\sm E$.
\item $\hat{h}(\xi)\geq0$, $\forall \xi\in\hat{G}$, and $\hat{h}(0)>0$.
\end{enumerate}
\end{defn}
The following theorem, which is also known as \emph{Delsarte's method}, is the main tool for estimating the maximal size of $\abs{B}$, where $A$ accepts pairwise orthogonal characters, restricted on $A$ and indexed by $B$.
The proof is adapted from \cite{Mat15}, in the setting of finite additive groups.
\begin{thm}\label{thm:linprog}
With notation as above, let $B\ssq G$ be such that $B-B\ssq \set{0}\cup (G\sm E)$, and let $h$ be a witness function for $E$. Then,
\[\abs{B}\leq\abs{G}\frac{h(0)}{\hat{h}(0)}.\]
\end{thm}

\begin{proof}
It holds
\[\hat{\bs 1}_B(\xi)=\sum_{b\in B}\xi(-b).\]
Put
\[S=\sum_{\xi\in\hat{G}}\abs{\hat{\bs 1}_B(\xi)}^2\hat{h}(\xi).\]
Since $\hat{h}$ is nonnegative, we obtain the estimate
\[S\geq \abs{\hat{\bs1}_B(0)}^2\hat{h}(0)=\abs{B}^2\hat{h}(0).\]
On the other hand,
\begin{align*}
S &=\sum_{\xi\in\hat{G}}\hat{\bs1}_B(\xi)\ol{\hat{\bs1}_B(\xi)}\hat{h}(\xi)\\
 &=\sum_{\xi\in\hat{G}}\bra{\sum_{b\in B}\xi(-b)\sum_{b'\in B}\xi(b')}\hat{h}(\xi)\\
&=\sum_{\xi\in\hat{G}}\sum_{b,b'\in B}\xi(b'-b)\hat{h}(\xi)\\
&=\abs{G}\sum_{b,b'\in B}h(b'-b)\\
&\leq \abs{G}\abs{B}h(0),
\end{align*}
using the Fourier inversion, the definition of a witness function and the hypothesis $B-B\ssq \set{0}\cup (G\sm E)$. Combining the two estimates above we obtain
\[\abs{B}^2\hat{h}(0)\leq S\leq \abs{B}\abs{G}h(0),\]
thus,
\[\abs{B}\leq\abs{G}\frac{h(0)}{\hat{h}(0)},\]
as desired.
\end{proof}
Our goal is to show that there can be no spectral set, as described in Theorem \ref{thm:charspec}\eqref{empty}, by finding a witness function $h$ with respect to $E=\supp(\wh{\bs1}_A)$ that satisfies
\[\frac{h(0)}{\hat{h}(0)}<\frac{\abs{A}}{p^3}=\frac{k}{p^2}.\]
We note that the function $h_A=\abs{\wh{\bs1}_A}^2$ attains equality, as 
\[\wh{\abs{\wh{\bs1}_A}^2}=\wh{\wh{\bs1_A*\bs1_{-A}}}=p^3\bs1_A*\bs1_{-A}\]
is a nonnegative function, and
\[\frac{h_A(0)}{\hat{h}_A(0)}=\frac{\abs{A}^2}{p^3\abs{A}}=\frac{\abs{A}}{p^3},\]
proving the trivial fact from linear algebra, that the number of pairwise orthogonal functions on $A$ are at most $\abs{A}$. The last equality holds true also for the functions $h_{\la A}$, $1\leq \la\leq p-1$, as well as for their average,
\[h=\sum_{\la=1}^{p-1}h_{\la A}.\]
The latter is constant on every punctured line, as
\[h(x)=\Tr_{\QQ(\ze_p)/\QQ}\abs{\wh{\bs1}_A(x)}^2, \;\;\; \forall x\in \ZZ_p^3.\]
Moreover, $h$ takes only integer values. A function that is constant on every punctured line in $\ZZ_p^3$ will be called \emph{balanced} (equivalently, it is a homogeneous function of degree 0). The following lemma shows us that it suffices to restrict our attention to balanced functions, in search for a witness function with desirable conditions, since the ``forbidden" set $\supp(\wh{\bs1}_A)$ is a union of lines through the origin.

\begin{lemma}\label{lem:balanced}
Suppose that $h$ is a witness function with respect to the forbidden set $E\ssq\ZZ_p^3$, which is a union of lines through the origin. Then, there is a balanced witness function $f$, also with respect to $E$, which satisfies $h(0)=f(0)$ and $\hat{h}(0)=\hat{f}(0)$.
\end{lemma}

\begin{proof}
Consider the function $f:\ZZ_p^3\to\RR$ which is defined by
\[f(x)=\frac{1}{p-1}\sum_{\la=1}^{p-1}h(\la x), \;\;\; \forall x\in\ZZ_p^3.\]
$f$ is obviously balanced, as $h$ is, and for every $x\notin E$ we have $\la x\notin E$ for $1\leq\la\leq p-1$, as $E$ is a union of lines. Therefore, $f(x)\leq0$, $\forall x\notin E$. We also have $\hat{f}\geq0$, as
\begin{align*}
\hat{f}(y) &=\frac{1}{p-1}\sum_{\la=1}^{p-1}\sum_{x\in\ZZ_p^3}h(\la x)\ze_p^\scal{y,-x}\\
&= \frac{1}{p-1}\sum_{\la=1}^{p-1}\sum_{x\in\ZZ_p^3}h(x)\ze_p^\scal{\la^{-1}y,-x}\\
&=\frac{1}{p-1}\sum_{\la=1}^{p-1}\hat{h}(\la^{-1}y)\geq0.
\end{align*}
Putting $y=0$ above we get $\hat{f}(0)=\hat{h}(0)$, and obviously satisfies $f(0)=h(0)$, completing the proof.
\end{proof}
Fourier transforms of balanced functions are also balanced; indeed,
\begin{equation}\label{eq:fourierbalanced}
\wh{\bs1}_0=\bs1_{\ZZ_p^3}, \;\;\; \wh{\bs1}_L=p\bs1_{L^\perp}, \;\;\; \wh{\bs1}_{L^*}=p\bs1_{L^\perp}-\bs1_{\ZZ_p^3}, 
\end{equation}
where $L$ is a line through the origin.

%% file: Sections/FourierProjective.tex
\bigskip\bigskip
\section{Fourier analysis on the finite projective plane}\label{FourierProjective}
\bigskip\bigskip
Now that we have restricted on balanced functions, we may work with functions on the finite projective plane of order $p$, denoted by $\bs P\FF_p^2$, along with the point $O$, which corresponds to the origin of $\ZZ_p^3$. A balanced function $f:\ZZ_p^3\to\CC$ corresponds to the function $\wt{f}:\bs P\FF_p^2\cup\set{O}\to\CC$, which is defined by
\[\wt{f}([x_1:x_2:x_3])=f(x_1,x_2,x_3), \;\;\; \forall x_1, x_2, x_3\in\ZZ_p.\]
$\wt{f}$ is well-defined since $f$ is balanced, and we note that we extend the notation on projective coordinates, by putting $O=[0:0:0]$. Since the Fourier transforms of balanced functions are also balanced, we define the Fourier transform for functions defined on $\bs P\FF_p^2\cup\set{O}$ so that it satisfies $\wt{\hat{f}}=\wh{\wt{f}}$. Similarly, for a set $D\ssq\ZZ_p^3$, that is a union of lines through the origin, we denote by $\wt{D}$ the subset of $\bs P\FF_p^2\cup\set{O}$ satisfying
\[\wt{D}=\set{[x_1:x_2:x_3]:(x_1,x_2,x_3)\in D}.\]
If $P=[x_1:x_2:x_3]\in\bs P\FF_p^2$, we denote by $P^\perp$ the set
\[P^\perp=\set{[y_1:y_2:y_3]\neq O:x_1y_1+x_2y_2+x_3y_3=0}.\]
The indicator function of a subset $K\ssq\bs P\FF_p^2\cup\set{O}$ will be denoted by $\de_K$, to make it clearer that we work on $\bs P\FF_p^2\cup\set{O}$ and not $\ZZ_p^3$. If $K$ consists of a single point $P$, then we write $\de_P$ instead of $\de_\set{P}$. Lastly, the constant function on $\bs P\FF_p^2\cup\set{O}$ which equals 1, is denoted simply by $\bs1$. The equations \eqref{eq:fourierbalanced} are then transformed to
\begin{equation}\label{eq:Fourierproj}
    \wh{\de}_O=\bs1, \;\;\; \wh{\de}_P=p\de_{P^\perp}+p\de_O-\bs1.
\end{equation}

When we pass from $\ZZ_p^3\sm\set{0}$ to $\bs P\FF_p^2$, punctured lines collapse to points and punctured planes collapse to lines; moreover, coplanar punctured lines correspond to collinear points on the projective plane.

Now consider a subset $A\ssq\ZZ_p^3$ having $pk$ elements, where $1<k<p$. As we have seen from Theorem \ref{thm:charspec}\eqref{empty}, if $A$ is spectral, then $Z(\wh{\bs1}_A)$ intersects every punctured plane, but contains none. Therefore, the set $Z=\wt{Z(\wh{\bs1}_A)}\ssq \bs P\FF_p^2$ intersects every projective line, but contains none.

\begin{defn}
A set $S\ssq \bs P\FF_p^2$ is called a \emph{blocking set} if it intersects every projective line, but contains none. $S$ is called a \emph{minimal} blocking set, if $S\sm\set{P}$ is not a blocking set, for every $P\in S$.
\end{defn}

So the set $Z$ is a blocking set, and so is the complement of every blocking set, by definition. We denote by $S$ the complement of $Z$ in $\bs P\FF_p^2$, or equivalently,
\begin{equation}\label{eq:suppwtnfcn}
S=\wt{\supp(\wh{\bs1}_A)}\sm\set{O}.
\end{equation}
The following bounds on the size of blocking sets were given by \cite{Blokhuis, blockingsets}. 
\begin{thm}\label{thm:blockingsets}
Let $S\ssq\bs P\FF_p^2$ be a blocking set. Then,
\[\frac{3}{2}(p+1)\leq\abs{S}\leq p^2-\frac{1}{2}p-\frac{1}{2}.\]
If $S$ is a minimal blocking set, then
\[\abs{S}<p\sqrt{p}+1.\]
\end{thm}

Our ultimate goal is to prove the \textbf{S-T} direction in $\ZZ_p^3$. If counterexamples exist, they are described by Theorem \ref{thm:charspec}\eqref{empty}. In view of Lemma \ref{lem:balanced}, it suffices to show that there is a balanced witness function $h$ with respect to $\supp(\wh{\bs1}_A)$, such that
\[\frac{h(0)}{\hat{h}(0)}<\frac{k}{p^2},\]
where $\abs{A}=pk$, and $Z(\wh{\bs1}_A)$ intersects every punctured plane, but contains none. Equivalently, it suffices to find a function $\wt{h}:\bs P\FF_p^2\cup\set{O}\to\RR$, such that $\wt{h}(P)\leq0$, $\forall P\in Z$, and $\wh{\wt{h}}\geq0$, such that
\[\frac{\wt{h}(O)}{\wh{\wt{h}}(O)}<\frac{k}{p^2}.\]
We will actually restrict to functions $\wt{h}$ that vanish on $Z$, so that $\wt{h}\geq0$ as well. Such functions (i.~e. functions $f$ such that $f, \hat{f}\geq0$) are called \emph{positive definite}.

%% file: Sections/Proof.tex
\bigskip\bigskip
\section{Proof of Theorem \ref{mainthm}}\label{Proof}
\bigskip\bigskip

Let $S$ be the set defined as in \eqref{eq:suppwtnfcn}, and let $S'\ssq S$ be a minimal blocking set\footnote{One could construct such a set by removing points from S. For example, we remove $P$ from $S$, if every line through $P$ contains also another point of $S$. The process continues until no such $P$ exists.}. We define
\[\wt{h}=\de_{S'}+(\abs{S'}-p)\de_O.\]
$\wt{h}$ obviously vanishes on $Z$, so to show that it is indeed a witness function with respect to $S$, we have to show that $\wt{h}$ is positive definite. Since $\wt{h}\geq0$, we only have to compute the Fourier transform using \eqref{eq:Fourierproj},
\[\wh{\wt{h}}=\sum_{P\in S'}p\de_{P^\perp}+p\abs{S'}\de_O-\abs{S'}\bs1+\abs{S'}\bs1-p\bs1=p\bra{\sum_{P\in S'}\de_{P^\perp}+\abs{S'}\de_O-\bs1}.\]
The fact that $\wh{\wt{h}}\geq0$ follows from the fact that $S'$ is a blocking set; indeed, as for an arbitrary $Q\in\bs P\FF_p^2$ it holds
\[\wh{\wt{h}}(Q)=p\bra{\sum_{P\in S'}\de_{P^\perp}(Q)-1}=p\bra{\sum_{P\in S'}\de_{Q^\perp}(P)-1}=p(\abs{S'\cap Q^\perp}-1)\geq0.\]
Finally, $\wt{h}(O)=\abs{S'}-p>0$, $\wh{\wt{h}}(O)=p(\abs{S'}-1)>0$, and
\[\frac{\wt{h}(O)}{\wh{\wt{h}}(O)}=\frac{\abs{S'}-p}{p(\abs{S'}-1)}.\]
When is the latter less than $k/p^2$? Putting $k=p-m$, this is equivalent to
\[\abs{S'}-1<\frac{p(p-1)}{m},\]
which is satisfied if
\[m<\frac{p-1}{\sqrt{p}},\]
as
\[\frac{p-1}{\sqrt{p}}=\frac{p(p-1)}{p\sqrt{p}}<\frac{p(p-1)}{\abs{S'}-1},\]
by virtue of Theorem \ref{thm:blockingsets}. This shows that if
\[p^2-p\sqrt{p}+\sqrt{p}<p(p-m)=\abs{A}<p^2,\]
then $A$ cannot be spectral, completing the proof of Theorem \ref{mainthm}.

%% file: Sections/Future.tex
\bigskip\bigskip
\section{Conclusion}\label{Future}
\bigskip\bigskip
Delsarte's method provided a partial result for Fuglede's conjecture in $\ZZ_p^3$. This can be improved further, using Lemma 3 from \cite{FMV19}; translated to the terminology of this paper, this shows that either a spectral set $A\ssq\ZZ_p^3$ with $\abs{A}=pk$, $1<k<p$ (if it exists!), is contained in $k$ parallel planes, or $\wt{h}$ is supported on at least $3$ points on every projective line. In the latter case, the set $S$ is a $3$-fold blocking set, and if $S'\ssq S$ is a minimal $3$-fold blocking set, we define
\[\wt{h}=\de_{S'}+(\abs{S'}-3p)\de_O.\]
Working exactly as in Section \ref{Proof}, we obtain that $A$ cannot be spectral if
\begin{equation}\label{eq:improved}
p\bra{p-\frac{3(p-1)}{\sqrt{3p-5}+1}}<\abs{A}<p^2,
\end{equation}
using the upper bound on the size of a $3$-fold blocking set in $\bs P\FF_p^2$, given by \cite{minimalblockingsets}
\[\frac{1}{2}p\sqrt{12p-20}+p.\]
This method has limitations; even if $S$ were a $t$-fold blocking set, and $S'\ssq S$ a minimal subset with the same property, then the witness function
\begin{equation}\label{eq:newwitness}
\wt{h}=\de_{S'}+(\abs{S'}-tp)\de_O
\end{equation}
would improve \eqref{eq:improved} using \cite[Theorem 1.1]{minimalblockingsets}, but not to the point where sets of cardinality $\abs{A}=2p$ are included, since $t\leq p-2$. We note that we cannot have $t\geq p-1$, since that would imply that no three points of $Z$ are collinear; if that were true, then considering all lines through a single point of $Z$ would give the upper bound $\abs{Z}\leq p+2$, which violates Theorem \ref{thm:blockingsets}, since $Z$ is also a blocking set.

If we are in the former case, where $S$ is not a $3$-fold blocking set, then the best we can do with this approach is to examine whether $S$ contains minimal blocking sets of small cardinality. However, even the lower bound described by Theorem \ref{thm:blockingsets}, would improve Theorem \ref{mainthm} to
\[p\cdot\frac{p^2+5p}{3p+1}<\abs{A}<p^2,\]
at best case scenario, which does not include $\abs{A}=kp$ for roughly $p/3$ values of $k$.

So, how can we approach uniformly the cases where $\abs{A}=kp$, with $k$ small? The limitations of the method described so far is that we consider only witness functions of the form \eqref{eq:newwitness}, which seems to tackle the cases where $\abs{A}$ is large. A cleverer selection of a witness function supported on $S\cup\set{O}$, perhaps in a probabilistic fashion, might give the answer to this question.